\documentclass[12pt,reqno]{amsart}
\usepackage[utf8]{inputenc}
\usepackage[T1]{fontenc}
\usepackage{amscd}
\usepackage{amssymb}
\usepackage{amstext}
\usepackage{amsmath}
\usepackage{amsfonts}
\usepackage{amsthm}
\usepackage{color}

\theoremstyle{plain}
\newtheorem{thm}{Theorem}[section]
\newtheorem{lem}[thm]{Lemma}

\theoremstyle{definition}
\newtheorem{definition}[thm]{Definition}
\newtheorem{ex}[thm]{Example}
\newtheorem{rem}[thm]{Remark}

\numberwithin{equation}{section}

\def\a{\mathcal{A}}
\def\f{\mathcal{F}}
\def\g{\mathcal{G}}
\def\h{\mathcal{H}}
\def\u{\mathcal{U}}
\def\z{\mathcal{Z}}

\def\nn{\mathbb{N}}
\def\zz{\mathbb{Z}}
\def\rr{\mathbb{R}}

\def\diam{{\rm diam}}
\def\exp{{\rm exp}}
\def\expf{{\rm exp}^{\f}}
\def\htop{h_{{\rm top}}}
\def\intt{{\rm int}}
\def\tub{{\rm Tub}}

\keywords{geometric entropy, leafwise Finsler structure}
\thanks{The research of the first author was partially supported by the Polish Ministry of Science and Higher Education and the
AGH grant No. 11.420.04\newline\indent
The research of the second author was supported by the National Science Center grant No. 6065/B/H03/2011/40 and the Grant-in-Aid for Scientific 
Research (S) 24224002 of the Japan Society for Promotion of Science.}
\address{Ilona Michalik\newline
AGH University of Science and
Technology,\newline Faculty of Applied Mathematics,\newline al. Mickiewicza 30,\newline 30-059 Krak\'ow,\newline Poland}
\email{imichali@wms.mat.agh.edu.pl}
\address{Szymon Walczak\newline
University of \L \'od\'z,\newline Faculty of Mathematics and Computer Science,\newline ul. Banacha 22,\newline 90-238 \L \'od\'z,\newline Poland}
\email{szymon.walczak@math.uni.lodz.pl}
\date{\today}

\title{Entropy of foliations with leafwise Finsler structure}
\author{Ilona Michalik, Szymon Walczak}

\begin{document}
\sloppy
\baselineskip18pt

\begin{abstract}
We extend the notion of the geometric entropy of foliation to foliated manifolds equipped with leafwise Finsler structure. We study the relation between the geometric entropy and the topological entropy of the holonomy pseudogroup. The case of foliated manifold with leafwise Randers structure. In this case the estimates for one dimensional foliation defined by a vector field in term of topological entropy of a flow are presented.
\end{abstract}

\maketitle

\section{Introduction}

The notion of the topological entropy was introduced by Adler, Konheim and McAndrew in 1965 in \cite{bib:AdlKonMcA65}. Another approach was presented by Bowen \cite{bib:Bow71} in early 70's. Ghys, Langevin and Walczak, in \cite{bib:GhyLanWal88}, extended this notion to the topological entropy for finitely generated groups and pseudogroups of continuous transformations, as well as the geometric entropy of foliation on a compact foliated Riemannian manifold. The entropy of foliation has more geometric nature, because it depends on a Riemannian metric chosen for foliated manifold. On the other hand, the dynamics of Finsler spaces have become a subject of a research of mathematicians in late 90's and recently. However, the research in this field is rather in the initial phase.

The aim of this paper is to extend the notion of the geometric entropy of foliations to the foliated manifolds equipped with leafwise Finsler structure. In paragraphs 2 and 3, one can find all necessary definitions and properties related to entropy and foliations with leafwise Finsler metric. Next paragraph describes relations between geometric and topological entropy. Fifth part of the paper refers to foliations with leafwise Randers norm. Last paragraph describes the entropy of one dimensional foliations defined by a unit vector field with leafwise Randers metric. 

\section{Leafwise Finsler structures}

Let us recall that a {\it Minkowski norm} on a vector space $V$ is a non-negative function $F:V\to[0,\infty)$ such that
\begin{enumerate}
\item $F$ is $C^{\infty}$ on $V\setminus\{0\}$,
\item $F(\lambda v)=\lambda F(v)$ for any $\lambda>0$ and $v\in V$,
\item for every $y\in V\setminus \{0\}$, the symmetric bilinear form
\[
  g_y (u,v):= \frac{1}{2} \frac{\partial ^2}{\partial t\partial s} F^2(y+su+tv)|_{t=s=0}
\]
is positively defined.
\end{enumerate}

Now, let $M$ be a smooth manifold. A function $F:TM\to [0,\infty)$ is called a {\it Finsler norm} if
\begin{enumerate}
\item $F$ is $C^{\infty}$ on the tangent bundle with removed the zero section $TM\setminus \{0\}$,
\item for any $x\in M$ the restricted norm $F_x=F|_{T_xM}$ is a Minkowski norm. 
\end{enumerate}
The pair $(M,F)$ is called a Finsler space.

\begin{ex}\label{ex:Randers norm}
Let $(M,g)$ be a Riemannian manifold, and let $\beta:TM\to\mathbb{R}$ be a $1$-form. Let $\alpha:TM\to [0,\infty)$ be the norm defined by $g$, that is, $\alpha(v)=\sqrt{g_x(v,v)}$ for all $v\in T_x M$. Suppose that the $g$-norm of $\beta$ satisfies $||\beta||_g < 1$. We set $F(v)=\alpha(v)+\beta(v)$. $F$ is a Finsler norm and it is called a {\it Randers norm}.
\end{ex}

Note that the Finsler norm induces a function $d:M\times M\to [0,\infty)$ by the formula
\[
d(x,y)=\inf_{\gamma} \int_0^1 F(\dot{\gamma}(t)) dt
\]
where the infimum is taken over all curves $\gamma:[0,1]\to M$ linking $x$ and $y$. Function $d$ is a quasi-metric, that is, 
\[
(d(x,y)=0 \textrm{ iff } x=y) \textrm{ and } d(x,y)+d(y,z)\geq d(x,z).
\]

Let $(M,\f,g)$ be a foliated Riemaniann manifold. Having $g$, we decompose the tangent bundle to the orthogonal sum of the bundle tangent to $\f$ and the orthogonal bundle, that is, $TM=T\f \oplus T\f^{\perp}$. We replace the norm induced in $T\f$ by the Riemannian structure $g|_{T\f}$ by a Finsler norm $F_{\f}$. Denote by $\pi_1:TM\to T\f$ and $\pi_2:TM\to T\f^{\perp}$ the natural projections. We set
\[
F(v)=\sqrt{F_{\f}^2(\pi_1(v))+g(\pi_2(v),\pi_2(v))}.
\]
$F$ is a Finsler norm on $TM$ and coincides with $\sqrt{g(v,v)}$ for $v\in  T\f^{\perp}$ and with $F_{\f}$ for $v\in T\f$. We call $F$ a \emph{leafwise Finsler structure} on $(M,\f)$.

\section{Geometric entropy of foliations with leafwise Finsler metric}

Let $(M,\f,F)$ be a foliated manifold with leafwise Finsler structure. Let $\u$ be a \textit{nice covering}, i.e., a covering by the domains $D_{\varphi}$ of the charts of a nice foliated atlas $\a$, that is an atlas satisfying
\begin{enumerate}
\item the covering $\{D_{\varphi}: \varphi\in\a\}$ is locally finite,
\item for any $\varphi\in\a$, the set $R_{\varphi}=\varphi(D_{\varphi})\subset \rr^n$ is an open cube,
\item if $\varphi,\psi\in\a$, and $D_{\varphi}\cap D_{\psi}\neq\emptyset$, then there exists a chart $\chi=(\chi',\chi'')$ such that for any leaf $L$ of $\f$ the connected components of $L\cap D_{\chi}$ are given by the equation $\chi''={\rm const}$, and $R_{\chi}$ is an open cube, $D_{\chi}$ contains the closure of $D_{\varphi}\cup D_{\psi}$ and $\varphi=\chi|_{D_{\varphi}}$ and $\psi=\chi|_{D_{\psi}}$.
\end{enumerate}
Let $U\in\u$. Equip the space of plaques $T_U=U/_{\f|U}$ with the quotient topology. The disjoint union $T=\bigsqcup \{T_U; U\in\u\}$ is called a {\it complete transversal} for $\f$. Note that each of $T_U$ can be mapped homeomorphically onto a $C^r$-submanifold $T'_U\subset U$ transverse to $\f$.

Following \cite{bib:GhyLanWal88}, let us recall that for a given nice covering $\u$ of $(M,\f)$ there exists an $\varepsilon_0>0$ such that any point $x\in U$, $U\in \u$, can be projected orthogonally in the unique way to the plaque $P_y\subset U$ passing through a point $y\in U$ if only $\max\{ d(x,y),d(y,x)\}<\varepsilon_0$. 

Let $\gamma:[0,1]\to L$ be a leafwise curve beginning at $x\in U$. For any $y\in U$ lying within the distance $\varepsilon<\varepsilon_0$, we can project orthogonally an initial part of the curve $\gamma$ to the plaque $P_y$ passing through $y$. Replacing $x$ and $y$ by the endpoints of the already projected piece and its image $\gamma_1$, we can continue this process as long as the distance between $\gamma$ and $\gamma'$ does not exceed $\varepsilon_0$. We will denote the projection of $\gamma$ by $p_y\gamma$.

Let $\u$ be a nice covering and let $T$ be the complete transversal for $\u$. Let $\varepsilon\in (0,\varepsilon_0)$, and let $d$ denotes the metric induced by the Finsler structure.

\begin{definition}
We say that $x,y\in T$ are $(R,\varepsilon)$-separated by $\f$ with respect to $F$ if either 
\begin{itemize}
  \item $\max\{d(x,y),d(y,x)\}\geq\varepsilon_0$
\end{itemize}
or 
\begin{itemize}
  \item there exists a leaf curve $\gamma:[0,1]\to L_x$ such that $\gamma(0)=x$, $l(\gamma)=\int_0^1 F(\dot{\gamma}(t))dt \leq R$ and 
  \[
  \max\{d(\gamma(1),p_y\gamma(1)),d(p_y\gamma(1),\gamma(1))\}\geq \varepsilon.
  \]
  (or a leaf curve $\gamma:[0,1]\to L_y$ such that $\gamma(0)=y$, $l(\gamma)\leq R$, and $\max\{d(\gamma(1),p_x\gamma(1)),d(p_x\gamma(1),\gamma(1))\}\geq \varepsilon$).
\end{itemize}

A subset $A\subset T$ is called $(R,\varepsilon)$-separated if any two points $x,y\in A$, $x\neq y$, are $(R,\varepsilon)$-separated. Let $s(R,\varepsilon,\f)$ denote the maximum cardinality of an $(R,\varepsilon)$-separated subset of $T$. We set $s(\varepsilon,\f) = \limsup\limits_{R\to\infty} \frac{1}{R} \log s(R,\varepsilon,\f)$, and 
\[
h(\f,F) = \lim_{\varepsilon\to 0^+} s(\varepsilon,\f).
\]
\end{definition}

\begin{rem}
The number $h(\f,F)$ does not depend on the choice of the nice covering $\u$. Let $\u'$ and $T'$ be another nice covering and complete transversal. Let $\varepsilon>0$ be small enough, and let us denote by $d_{\f}$  the leafwise metric induced by the Finsler structure $F_{\f}$. Since $M$ is compact, the geometry of $\f$ is bounded. Hence, one can project $T$ onto $T'$ in such a way that any $(R,\varepsilon)$-separated points $x,y\in T$ are projected to $x',y'\in T'$, respectively, which are $(R+R_0,\varepsilon)$-separated with $R_0$ being the maximum of the numbers $d_{\f}(x,x')$ and $d_{\f}(x',x)$, $x\in T\cap U$, $x'\in T'\cap U'$, $U\in\u$, $U'\in \u'$, and the plaques $P_x\subset U$ and $P_{x'}\subset U'$ intersects. Thus
\[
  s'(R-R_0,\varepsilon,\f)\leq s(R,\varepsilon,\f)\leq s'(R+R_0,\varepsilon,\f),
\]
and both numbers $h(\f,F)$ and $h'(\f, F)$ are equal. 
\end{rem}

\begin{rem}
Since any two Riemannian structures $g$ and $g'$ on a compact manifold satisfies
\[
  c^{-2} g(v,w) \leq g'(v,w)\leq c^2 g(v,w)
\]
for some constant $c>1$, then the number $h(\f,F)$ does not depend on the choice of the Riemannian part of $F$. Indeed, there exists a constant $a>1$ such that for any leaf curve $\gamma$ and its orthogonal projections $p_y\gamma$ and $p_y'\gamma$, with respect to $g$ and $g'$ respectively, satisfies
\[
  d(\gamma(t),p_y\gamma(t)) \leq a\cdot d'(\gamma(t),p_y'\gamma(t)),
\]
if $d(\gamma(t),p_y'\gamma(t))<\varepsilon$ for sufficiently small $\varepsilon>0$. Thus any two $(R,\varepsilon)$-separated points with respect to $F = \sqrt{F_{\f}^2+g}$ are $(R,\frac{\varepsilon}{a})$-separated with respect to $F' = \sqrt{F_{\f}^2+g'}$, and $ h(\f,F) \leq h(\f,F')$. Analogically we show that $h(\f,F') \leq h(\f,F)$.
\end{rem}

Since any two leafwise Finsler structures $F$ and $F'$ on a compact foliated manifold satisfies
\[
 \frac{1}{c} F(v) \leq F'(v) \leq c\cdot F(v)
\]
for some constant $c\geq 1$, the geometric entropies $h(\f,F)$ and $h(\f,F')$ are both either equal to zero or not. The number $h(\f,F)$ is called the {\it geometric entropy of foliation with leafwise Finsler structure}. In further consideration we will denote by $F$ both, the structure $F_{\f}$ and the leafwise Finsler structure $F=\sqrt{F_{\f}^2+g}$.

\section{Relation between geometric entropy and topological entropy of holonomy pseudogroup}

Let $(M,\f)$ be a compact foliated manifold. Following \cite{bib:GhyLanWal88} or \cite{bib:Wal04}, one can define the topological entropy of the holonomy pseudogroup $\h_{\u}$ defined by the nice covering $\u$. The symbol $D_f$ denotes here the domain
of a map $f$.

To begin, let $\g$ be a pseudogroup (see \cite{bib:Wal04}) on a metric (quasi-metric) space $(X,d)$ generated by a good symmetric finite set $\g_1$, that is
\begin{enumerate}
\item for any $g\in\g$ and any $x\in D_g$ there exist $g_1,\dots,g_n\in \g_1$ and open subset $U\subset D_g$ containing $x$ such that
\[
  g|_U = g_1\circ\cdots\circ g_n|_U,
\]
\item for any $g\in \g_1$ the exits a compact set $K_g\subset D_g$ such that $g|_{\intt K_g}$ generate $\g$.

\end{enumerate}

We say that $x,y\in X$ are $(n,\varepsilon)$-separated by $\g$ if there exists 
\[
g\in \g_n^c := \{ g_1|_{K_1}\circ \cdots \circ g_n|_{K_n}; g_i\in \g_1\}
\]
such that $\{x,y\}\subset D_g$ and
\[
\max\{d(g(x),g(y)), d(g(y),g(x))\}\geq \varepsilon.
\]
A subset $A$ of $X$ is called $(n,\varepsilon)$-separated if any two distinct points of $A$ are $(n,\varepsilon)$-separated. Let $s(n,\varepsilon,\g_1)$ be the maximal cardinality of an $(n,\varepsilon)$-separated subset of $X$. We set 
\[
s(\varepsilon,\g_1) = \limsup_{n\to\infty} \frac{1}{n} \log s(n,\varepsilon,\g_1).
\]
The number $h(\g,\g_1) = \lim\limits_{\varepsilon\to 0^{+}} s(\varepsilon,\g_1)$ is called the {\it topological entropy of the pseudogroup $\g$ with respect to $\g_1$}.

Let $\u$ be a nice covering of $(M,f)$ and let $T$ be a complete transversal. Given two sets $U,V\in\u$ such that $U\cap V\neq\emptyset$, one can define the {\it holonomy map} $h_{VU}:D_{VU}\to T_V$ with $D_{VU}$ being the open subset of $U$ consisting of all plaques $P\subset U$ such that $P\cap V\neq\emptyset$ by
\[
h_{VU}(P) = P'\textrm{ iff }P\subset U\textrm{ and }P'\subset V\textrm{ insterect.}
\]
The mappings $h_{VU}$ generates the {\it holonomy pseudogroup} $\h_{\u}$ on $T$. We will denote by $\h_\u^1$ the set of $\{h_{VU}\}$ of the generators of $\h_\u$.

One of the main results of \cite{bib:GhyLanWal88} is Theorem 3.4 about the relation between the geometric entropy $h(\f,g)$ of a foliation on a Riemannian manifold and the topological entropy of the holonomy pseudogroup $\h_{\u}$ defined by the nice covering $\u$. We extend this result to the class of foliations with leafwise Finsler structures. 

Let $(M,\f,F)$ be a foliated manifold equipped with leafwise Finsler structure.

\begin{thm}\label{thm:GeometricVsPseudogroup} 
Let $U$ be a nice covering, and let $\diam(\u)$ be  the diameter of the nice covering $\u$, that is, 
\[
\diam(\u) = \max_{U\in \u}\; \max_{P\subset U}\; \max_{x,y\in P} d_{\f}(x,y),
\]
where $P$ denotes a plaque of a chart $U$, and $d_{\f}$ is the leafwise distance defined by $F$. Then
\[
h(\f,F) = \sup_{\u} \left\{ \frac{1}{\diam(\u)} h(\h_{\u}, \h_{\u}^1) \right\}
\]
\end{thm}

We here repeat the proof of Theorem 3.4 of \cite{bib:GhyLanWal88} with the necessary changes due to the fact that the metric induced by the Finsler structure is asymmetric. In the proof, $\Lambda = \max_{v\in T \f\setminus \{0\} }\; \frac{F(v)}{F(-v)}$.

\begin{lem}\label{lem:GeometricVsPseudogroupLemma1}
For $\Delta$ and $\rho$ small enough, there exists $\tilde\beta>0$ such that $\rho>\tilde\beta$ and the following is satisfied:

Let $x_1,x_2$ be two points lying on the same leaf $L$ and such that $d_{\f}(x_1,x_2)= \frac{2\Delta}{\Lambda} - \alpha$ for some $\alpha>0$. Let $y_1,y_2$ be two points of transversals $T_1$ and $T_2$ passing through $x_1$ and $x_2$, respectively, and lying on the same plaque with diameter not exceeding $4\Delta$. If
\[
\max\{ d(x_1,y_1), d(y_1,x_1), d(x_2,y_2), d(y_2,x_2)\}<\tilde\beta,
\]
then $d_{\f}(y_1,y_2)\leq \frac{2\Delta}{\Lambda} - \frac{\alpha}{2}$.
\end{lem}
\begin{proof}
Let $\gamma:[0,1]\to L$ be a curve linking $x_1$ with $x_2$ such that $l(\gamma)=\frac{2\Delta}{\Lambda}-\alpha$. Let $p_{y_1}\gamma$ be the orthogonal projection of $\gamma$ onto the plaque $P_{y_1}$. Since $M$ is compact, there exists $\tilde\beta>0$ such that $|l(\gamma)-l(p_{y_1}\gamma)|<\frac{\alpha}{4}$, and $\{d_{\f}(y_2,p_{y_1}\gamma(1)), d_{\f}(p_{y_1}\gamma(1),y_2)\}<\frac{\alpha}{4}$. Thus
\[
d_{\f}(y_1,y_2)\leq l(p_{y_1}\gamma)+d_{\f}(y_2,p_{y_1}\gamma(1))+d_{\f}(p_{y_1}\gamma(1),y_2)\leq \frac{2\Delta}{\Lambda} - \frac{\alpha}{2}.
\]
This ends our proof.
\end{proof}

Let $\rho>0$, and let $S_x = \exp B^{\perp}(0_x,2\rho)$ be the image in the exponential map $\exp$ on $M$ where $B^{\perp}$ is a ball centered in $0_x$ and contained in the orthogonal complement $T_x\f^{\perp}$ of $T_x\f$. Set
\[
T_x = \exp B^{\perp} (0,\rho),\quad U_x = \bigcup_{y\in T_x} B_F\left(y, \frac{\Delta}{\Lambda}\right).
\] 

\begin{lem}\label{lem:GeometricVsPseudogroupLemma2}
Let $\z=\{z_1,\dots,z_N\}$ be a $\beta$-dense subset of $M$, $\beta<\frac{\tilde\beta}{10}$. Let $x_1$ and $x_2$ be two points of the same leaf with $d_{\f}(x_1,x_2)<\frac{2\Delta}{\Lambda} - \alpha$. Let $z\in\z$ (resp. $z'\in\z$) be an $\beta$-close point of $x_1$ (resp. $x_2$). Then the subsets $U_z$ and $U_{z'}$ have the following property:

If $\xi_1\in T_z$ and $\xi_2\in T_{z'}$ lie on the same plaque with diameter not exceeding $4\Delta$ and
\[
 \max\{ d(\xi_1,x_1), d(x_1,\xi_1), d(\xi_2,x_2), d(x_2,\xi_2)\}<\beta,
\]
then the minimal leaf geodesic in $L_{\xi_1}=L_{\xi_2}$ linking $\xi_1$ and $\xi_2$ is contained in the sum
\[
B_F\left(\xi_1,\frac{\Delta}{\Lambda}\right)\cup B_F(\xi_2,\Delta).
\]

\end{lem}
\begin{proof}
By Lemma \ref{lem:GeometricVsPseudogroupLemma1}, $d_{\f}(\xi_1,\xi_2)< \frac{2\Delta}{\Lambda} -\frac{\alpha}{2}$. So, there exists a curve $\gamma:[0,1]\to L_{\xi_1}$ such that $\gamma(0)=\xi_1$, $\gamma(0)=\xi_2$, and $l(\gamma) = d_{\f} (\xi_1,\xi_2)$. Let $t\in[0,1]$ be such a number that $d_{\f}(\xi_1,\gamma(t))=\frac{\Delta}{\Lambda}$. Then $d_{\f}(\gamma(t),\xi_2)\leq\frac{\Delta}{\Lambda}$. Since $d_{\f}$ is asymmetric then $d_{\f}(\xi_2, \gamma(t))\leq\Delta$.
\end{proof}

\begin{proof}[Proof of Theorem \ref{thm:GeometricVsPseudogroup}] To begin, let $\epsilon>0$, and $x,y \in T_U$, $U\in \u$, be $(n,\varepsilon)$-separated with respect to $h(\h_{\u}, \h^1_{\u})$. Then there exists a chain of maps $(U_1,\dots,U_n)$ such that the corresponding chains of plaques $(P_1,\dots, P_n)$ and $(Q_1,\dots,Q_n)$ with $x\in P_1$, $y\in Q_1$, $P_i, Q_i\in U_i$, $P_i\cap P_{i+1}\neq\emptyset$, $Q_i\cap Q_{i+1}\neq\emptyset$ satisfy 
\[
\max\{d(x_n, y_n), d(y_n, x_n)\}\geq \varepsilon
\]
where $x_n\in P_n\cap T_{U_n}$ and $y_n\in Q_n\cap T_{U_n}$ are the images in the holonomy map determined by $(U_1,\dots,U_n)$ of $x$ and $y$, respectively. Let $x_0=x$ and let us choose points $x_i\in P_i\cap P_{i+1}$, $i=1,\dots, n-1$. Link the points $x_i$ and $x_{i+1}$ by a leaf geodesic $\gamma_i$, $i=1,\dots, n-1$. The length of every $\gamma_i$ is smaller than $\diam(\u)$, and the length of a curve $\gamma$ built of $\gamma_i$'s and linking $x_0$ with $x_n$ is smaller than $n\cdot \diam(\u)$.  Shortening $\gamma$, if necessary, we can assume that the distance between $\gamma$ and its orthogonal projection $p_y\gamma$ is always smaller than $\varepsilon_0$, and the whole $\gamma$ can be projected to $L_y$. 

Since $T=\bigsqcup T_U$ is compact, there exists a constant $C>0$ such that 
\[
  \frac{1}{C} d(z,w) \leq d(z,p(z)) \leq C d(z,w)
\]
and 
\[
  \frac{1}{C} d(w,z) \leq d(p(z),z) \leq C d(w,z)
\]
if only $z,w\in T_U$, $U\in \u$, and $p(z)$ is the orthogonal projection of $z$ to the plaque $P_w$ passing through $w$. Hence $d(\gamma(1), p_y\gamma(1)) \geq \frac{\varepsilon}{C}$. This gives that $x$ and $y$ are $(n\cdot\diam(\u),\frac{\varepsilon}{C})$-separated with respect to $\f$. Thus,
\[
s(n,\varepsilon, \h_{\u}^1) \leq s\left(n\cdot \diam(\u),\frac{\varepsilon}{C},\f\right)
\]
for all $n\in\nn$, and $\varepsilon\in (0,\varepsilon_0)$. Finally,
\[
h(\h_{\u}, \h_{\u}^1) \leq \diam(\u)\cdot  h(\f,F).
\]

Let $\eta>0$, and $\Delta>0$ be such that the leafwise exponential mapping $\expf$ maps the balls $B^{\f}(0_x,4\Delta)$, where $B^{\f}(0_x,r)=\{v\in T_x\f: F(v)< r\}$, diffeomorphically onto strictly convex balls 
\[
B_F(x,4\Delta)=\{y\in L_x: d_{\f}(x,y)< 4\Delta\},\quad x\in M.
\]
Note that for small enough $\rho$ and $\Delta$, the sets $U_x$ are the domains of distinguished charts, and for any plaque $P\subset U_x$, the diameter $\diam(P_x)\leq (1+\frac{1}{\Lambda})\Delta$.

Let $\u_{\Delta}= \{U_z,z\in\z\}$. We may assume that the closures $\bar U_z$ and $\bar U_{z'}$, $z,z'\in \z$, overlap if only $\bar U_z$ and $\bar U_{z'}$ do. Thus $\u_\Delta$ is a nice covering of $(M,\f)$. Moreover, $\diam(\u_{\Delta})\leq (1+\frac{1}{\Lambda})\Delta$.

Let $\varepsilon>0$, and let $x,y$ be such that 
\[
\max\{ d(x,y), d(y,x)\}\leq\varepsilon
\]
and additionally they are $(R,\varepsilon)$-separated by $\f$ with respect to $F$. Hence, there exists a curve $\gamma:[0,R]\to L_{x}$ starting at $x$ with $l(\gamma)\leq R$ and such that $p_y\gamma$ is well defined on $[0,r]$, $r<R$, and $\max\{ d(\gamma(r),p_y\gamma(y)), d(p_y\gamma(y),\gamma(r))\}\geq\varepsilon$. Let us assume that $R=(1+\frac{1}{\Lambda})(1-\eta)n\Delta$, and let $x_k=\gamma(\frac{kr}{n})$, $k=0,\dots,n$. For each $x_k$ let us find a point $z_k\in\z$ which is $\beta$-close (see Lemma \ref{lem:GeometricVsPseudogroupLemma2}). 

The charts $(U_{z_0},\dots,U_{z_n})$ form a chain along $\gamma|_{[0,r]}$, and the corresponding holonomy map $h\in\h_{\u_{\Delta}}$ is well defined on the plaques $P,Q\in U_{z_0}$ containing $x$ and $y$, respectively. Moreover,
\[
\max\{d(h(Q),h(P)), d(h(P),h(Q))\}\geq C\cdot \varepsilon,
\]
where $C$ is the constant from the first part of this proof. We deduce that
\[
s\left(\left(1+\frac{1}{\Lambda}\right)(1-\eta)n\Delta,C\varepsilon,\f\right)\leq N(\varepsilon)\cdot s(n,\varepsilon,\h_{\u_{\Delta}}),
\]
with $N(\varepsilon)$ being the minimal cardinality of a covering of $M$ by balls of radius $\varepsilon$. Therefore,
\[
s(C\varepsilon,\f) \leq \frac{1}{(1+\frac{1}{\Lambda})(1-\eta)\Delta} s(\varepsilon,\h_{\u_{\Delta}}).
\]
Passing with $\eta$ to zero, we obtain
\[
h(\f,F)\leq \frac{1}{(1+\frac{1}{\Lambda})\Delta} h(\h_{\u_{\Delta}},\h_{\u_{\Delta}}^1)\leq \frac{1}{\diam(\u)} h(\h_{\u_{\Delta}},\h_{\u_{\Delta}}^1).
\]
This ends the proof.
\end{proof}

\section{Foliations with leafwise Randers norm}

Let $(M,\f,g)$ be a foliated Riemannian manifold. Let $F$ be a leafwise Randers norm, that is the norm given on leaves by
\[
  F(v) = \sqrt{g(v,v)}+\beta(v),\quad v\in T\f.
\]
Let $\|\beta\| = \max\limits_{v\in T^1_g\f} \beta(v)$. Let us suppose, similarly as in Example \ref{ex:Randers norm}, that $\|\beta\|<1$.

\begin{thm}\label{thm:Randers entropy}
The following inequalities hold:
$$\frac{1}{1+\|\beta\|}h(\f,g) \leq h(\f,F) \leq \frac{1}{1-\|\beta\|} h(\f,g).$$
\end{thm}
\begin{proof}
Let $G(v)=\sqrt{g(v,v)}$. Since $F(v) = G(v)+\beta(v)$ then for any $v\in T\f$
\begin{equation}\label{eq:Randers entropy}
F(v)\leq G(v)+\|\beta\|G(v)\textrm{ and } G(v)\leq F(v)+\|\beta\|G(v).
\end{equation}
Let $x,y$ be $(R,\varepsilon)$-separated with respect to $g$. So, there exists a curve $\gamma:[0,1]\to L_x$ such that $\gamma(0)=x$, $l_G(\gamma)\leq R$ and 
\[
d(\gamma(1),p_y\gamma(1))\geq \varepsilon.
\]
Using the first inequality in (\ref{eq:Randers entropy}), we obtain
\begin{align*}
l_{F}(\gamma) &= \int_0^1 F(\dot{\gamma}(t))dt \leq \int_0^1 G(\dot{\gamma}(t))dt + \int_0^1 \|\beta\|G(\dot{\gamma}(t))dt\\
&\leq R+\|\beta\|R = (1+\|\beta\|)R.
\end{align*}
Thus $x,y$ are $((1+\|\beta\|)R,\varepsilon)$-separated with respect to $F$. Hence,
\[
s(R,\varepsilon,g) \leq s((1+\|\beta\|)R,\varepsilon,F),
\]
\[
\frac{1}{R} \log s(R,\varepsilon,g) \leq \frac{1+\|\beta\|}{1+\|\beta\|}\frac{1}{R} \log s((1+\|\beta\|)R,\varepsilon,F),
\]
\begin{align*}
&\limsup_{R\to\infty}\frac{1}{R} \log s(R,\varepsilon,g) \\
&\leq (1+\|\beta\|) \limsup_{R\to\infty} \frac{1}{(1+\|\beta\|)R}\log s((1+\|\beta\|)R,\varepsilon,F),
\end{align*}
\[
s(\varepsilon,\f,g) \leq (1+\|\beta\|) s(\varepsilon,\f,F).
\]
Finally, $h(\f,g)\leq (1+\|\beta\|) h(\f,F)$.

The second inequality follows directly from the second inequality in (\ref{eq:Randers entropy}) and from the fact that every two points which are $(R,\varepsilon)$-separated with respect to $F$ are $(\frac{R}{1-\|\beta\|},\varepsilon)$-separated with respect to $g$.
\end{proof}

\section{Topological entropy of one dimensional foliation}

We will now recall the definition (following \cite{bib:Bow71} and \cite{bib:SayMolMog2014}) of the topological entropy of a uniformly continuous map on a quasi-metric space. 

Let $f:X\rightarrow X$ be a uniformly continuous transformation of a quasi-metric space $X$, that is, for any $\varepsilon>0$  and any $x\in X$ there exists $\delta>0$ such that for any $y\in X$
\[
\max\{d(x,y), d(y,x)\}<\delta \Rightarrow \max\{(d(f(x),f(y)), d(f(y),f(x))\}<\epsilon.
\]

For any $n\in\mathbb{N}$ and $x,y\in X$ let 
\[
d_n(x,y)=\max\limits_{0\leqslant k\leqslant n-1}\{\max\{d(f^k(x),f^k(y)), d(f^k(y),f^k(x))\}\},\; k\in\mathbb{N} .
\]
Let $n\in \mathbb{N}$ and $\varepsilon>0$. A subset $A$ of $X$ is said to be \emph{$(n,\varepsilon)$-separated} if $d_n(x,y)>\varepsilon$ for every $x,y\in A, x\neq y$. A set $B\subset X$ is said to \emph{$(n,\varepsilon)$-span} another set $K$ if for every $x\in K$ there is $y\in B$ such that $d_n(x,y)\leq\varepsilon$.

We set $s(n,\varepsilon, K)=\max\{\#A:A\subset K \;\textrm{is}\; (n,\varepsilon)-\textrm{separated}\}$, and $r(n,\varepsilon, K)=\min\{\#A:A\subset X \textrm{ is } (n,\varepsilon)-\textrm{spanning } K\}$.

\begin{lem}\label{lem:SeparatedVsSpanning}
The following inequalities hold
\begin{enumerate}
\item $r(n,\varepsilon,K)\leq s(n,\varepsilon,K) \leq r(n,\frac{\varepsilon}{2},K)<\infty$,
\item for $\varepsilon'<\varepsilon$
\[
  r(\varepsilon',K) \geq r(\varepsilon,K) \textrm{ and } s(\varepsilon',K) \geq s(\varepsilon,K).
\]
\end{enumerate}
\end{lem}
\begin{proof}
If $A$ is maximal $(n,\varepsilon)$-separated subset of $K$, then $A$ also $(n,\varepsilon)$-spans $K$. Thus $r(n,\varepsilon,K) \leq s(n,\varepsilon,K)$.

Let $A\subset K$ be an $(n,\varepsilon)$-separated set and let $B$ $(n,\frac{1}{2}\varepsilon)$-spans $K$. For any $x\in K$, there exists $g(x)\in B$ such that $d_n(x,g(x)) <\frac{\varepsilon}{2}$. Moreover, if $g(x)=g(y)$ then 
$d_n(x,y) <\varepsilon$.
Thus $g$ is injective on $A$ (since $A$ is $(n,\varepsilon)$-separated), and 
$s(n,\varepsilon,K) \leq r(n,\frac{\varepsilon}{2},K)$.

As $f$ is uniformly continuous on $(X,d)$ there is a $\delta>0$ such that 
$d_n(x,y) <\frac{\varepsilon}{2}$ if only $d(x,y)<\delta$ and $d(y,x)<\delta$. Thus $r(n,\frac{\varepsilon}{2},K)$ does not exceed the number of $\delta$-balls $B_{\delta}(z) = \{z'\in X: d(z,z')<\delta\textrm{ and }d(z',z)<\delta\}$ needed to cover $K$. So, $r(n,\frac{\varepsilon}{2},K)$ is finite, as $K$ is compact.

The inequalities in (2) are obvious.
\end{proof}

Finally, we define
\[
s(\varepsilon,K)=\lim_{n\rightarrow\infty}\sup\frac{1}{n}\log s(n,\varepsilon, K)
\]
and
\[
r(\varepsilon,K)=\lim_{n\rightarrow\infty}\sup\frac{1}{n}\log r(n,\varepsilon, K).
\]

\begin{definition}\label{def:TopologicalEntropy}
For any uniformly continuous map $f:X\to X$ on a quasi-metric space $(X,d)$ and any compact set $K\subset X$ define
\[
\htop(f,K)=\lim_{\varepsilon\rightarrow 0^{+}}s(\varepsilon,K)=\lim_{\varepsilon\rightarrow 0^{+}}r(\varepsilon,K)
\]
and
\[
\htop(f)=\sup_{K\textrm{  compact}} \htop(f,K).
\]
The number $\htop(f)$ is called the \emph{topological entropy of $f$}.
\end{definition}

Let us now study the geometrical entropy of a foliation given by the integral curves of a vector field $X$ on a compact manifold $M$. Note that any Finsler norm on 1-dimensional vector space is a Randers norm. Indeed, let
\[
 G(v) = \frac{1}{2}(F(v)+F(-v)) \quad, \beta(v) = \frac{1}{2}(F(v)-F(-v)).
\]
Then $G$ is a norm associated with a inner product $g$, while $\beta$ is a $1$-form. Moreover, $F(v) = G(v)+\beta(v)$ and $\|\beta\|_G < 1$.

Let $F=G+\beta$ be a leafwise Randers norm for which $X$ is a $G$-unit vector field, that is $G(X(p))=1$ for all $p\in M$. Let $\varphi=(\varphi_t:M\to M)_{t\in\rr}$ denote the flow of $X$. We recall \cite{bib:Wal04} that the topological entropy of a flow is equal to $\htop(\varphi_1)$.

\begin{thm}  If $\Lambda = \max\limits_{v\in T\f\setminus\{0\} } \frac{F(v)}{F(-v)}$ then
\[
\left(1+\frac{1}{\Lambda}\right) \htop (\varphi)\leq h(\f, F) \leq (1+\Lambda)\htop (\varphi).
\]
\end{thm}
\begin{proof}
 By Theorem 3.4.3 in \cite{bib:Wal04}, $h(\f,g) = 2\htop(\varphi)$. We have
 \begin{align*}
  h(\f,F)& \leq \frac{1}{1-\|\beta\|} h(\f,g) \leq \frac{2}{1-\|\beta\|} \htop(\varphi) \\
  &\leq \left(1+ \frac{1+\|\beta\|}{1-\|\beta\|}\right) \htop(\varphi).
 \end{align*}
 However, $\Lambda = \frac{1+\|\beta\|}{1-\|\beta\|}$. This gives the second inequality.
 
 To prove the first inequality, it is enough to observe that
 \begin{align*}
  h(\f,F)& \geq \frac{1}{1+\|\beta\|} h(\f,g) \geq \frac{2}{1+\|\beta\|} \htop(\varphi) \\
  &\geq \left(1+ \frac{1-\|\beta\|}{1+\|\beta\|}\right) \htop(\varphi).
 \end{align*}
 This ends our proof.
\end{proof}

\begin{thm}\label{thm:constantWind}
 Let $F=G+\beta$ be a leafwise Randers metric along one dimensional foliation defined by a $G$-unit vector filed $X$. Suppose that $\beta(X) = {\rm const}$. Then
 \[
 h(\f,F) = \frac{2}{1-\beta^2(X)} \htop(\varphi),
 \]
where $(\varphi_t)_{t\in\rr}$ denotes the flow of $X$.
\end{thm}
\begin{proof}

Let $a=\beta(X)$. Since $\|\beta\|_G<1$ then $a\in (-1,1)$. Let $A$ be an $(n,\epsilon)$-separated set by $\f$ with respect to $F$. Then the set $B=\varphi_{-\frac{n}{1-a}}(A)$ is $(\frac{2n}{1-a^2},\epsilon)$-separated with respect to $\varphi$. Hence
\[
 s(n,\epsilon,F)\leq s\left(\frac{2n}{1-a^2},\epsilon,\varphi\right).
\]
This gives
\[
 h(\f,F) \leq \frac{2}{1-\beta^2(X)} \htop(\varphi).
\]

Let $\eta>0$. Let us consider the fiber bundle $\pi:\tilde M_\eta \to M$ built  of orthogonal balls $B^{\perp}(0_x,\eta)\subset T_x\f^{\perp}$, $x\in M$. For every $x\in M$, let $\tub_\eta(x) = \varphi^{-1}_{x *} \tilde M_\eta$ be the bundle over $\rr$ induced by a map $\varphi_x:t\mapsto \varphi_t(x)$. It is known, that for $\eta$ small enough, the exponential map on $M$ defines a natural immersion $\iota_x:\tub_{\eta}(x)\to M$, and one can equip $\tub_{\eta}(x)$ with the induced leafwise Randers structure and with the induced vector field $\tilde X$, which generates a local flow $(\tilde \varphi_t)$. As mentioned in \cite{bib:Wal04}, the family $\pi_x^{-1}(s)$, where $\pi_x$ is a fiber bundle projection in $\tub_\eta(x)$ and $s\in \rr$, of fibers of $\tub_\eta(x)$ is not invariant under the flow $(\tilde\varphi_t)$.

Let us fix $\varepsilon>0$. Since $M$ is compact, the family $\pi^{-1}(s)$, $s\in \rr$, of fibers of $\tub_\eta(x)$ satisfies the following:

For any $\tau\in (0,1)$ there exists $\eta>0$ such that for any $x\in M$ and $y\in\tub_\eta(x)\cap \pi^{-1}(0)$ with defined local flow $(\tilde\varphi_t)$   and $\pi(\tilde\varphi_t(y))=1$ (respectively $\pi(\tilde\varphi_t(y))=-1$) we have $t>\tau$ ($t<-\tau$). Moreover, if $\tau$ and $\eta$ are as above, then $t\geq n\tau$ ($t\leq -n\tau$) whenever $(\tilde\varphi_t)$ is defined and $\pi(\tilde\varphi_t(y))=n$ (respectively, $-n$), $n\in\nn$.

Let us decompose $\tub_\eta (x)$ into the cylinders $C_n(x) = \pi_x^{-1}([(2n-1)\varepsilon,(2n+1)\varepsilon])$, $n\in\zz$. Since $\varepsilon$ is fixed, there exists $\eta$ independent of $x\in M$ such that the sets $\tilde\varphi_1(C_0(x))$ and $\tilde\varphi_{-1}(C_0(x))$ intersect at most three cylinders of the form $C_n(x)$. For every $y\in C_0(x)$, we consider the sequences $(n_k)_{k\in\nn}$ of integers such that $\tilde\varphi_k(y)\in C_{n_k}(x)$ (we set $\infty$ if $\tilde\varphi_k(y)$ is undefined). The number of such sequences of length $[2n/(1-a^2)]-1$ do not exceed $3^{k_a n}$ for some natural number $k_a$. So, we can decompose all cylinders $C_0(x)$ into the unions of sets $C_0(x) = K_1(x)\cup\cdots\cup K_{m(x)}(x)$, $m(x)\leq 3^{k_a n+2}$ satisfying the following:
\vskip10pt
\noindent($*$) If $y,z\in K_j(x)$ , $[\frac{n}{1-a}]\leq k\leq [\frac{n}{1+a}]$ and $\tilde\varphi_k(y)$ and  $\tilde\varphi_k(z)$ are defined, then $\tilde\varphi_k(z)$ and $\tilde\varphi_k(y)$ belongs to the same cylinder $C_{n_k}(x)$.
\vskip10pt
Let $A\subset T$ be a maximal $(n,\frac{\eta}{3})$-separated by $\f$ with respect to $F$, that is, $\sharp A = s(n,\frac{\eta}{3},F)$. Since $A$ is maximal, then it is $(n,\frac{\eta}{3})$-spanning for $T$, and the sets
\begin{align*}
  A(x) &= \{y\in T: \sup_{-[\frac{n}{1-a}]\leq t\leq [\frac{n}{1+a}]} d(\varphi_t(x),p_y\varphi_t(x))\leq\frac{\eta}{3}\\
  &\textrm{ and } 
  \sup_{-[\frac{n}{1-a}]\leq t\leq [\frac{n}{1+a}]} d(p_y\varphi_t(x),\varphi_t(x))\leq\frac{\eta}{3}\}, \quad x\in A
\end{align*}
cover $T$. Moreover, $\max\limits_{x\in A} \diam A(x) \leq \frac{2\eta}{3}$. Therefore, $A(x)\subset \iota_x(C_0(x))$, and we can decompose $C_0(x)$ into $K_j(x)$ and choose one point $y^x_j$ in each nonempty piece of $A(x)\cap K_j(x)$. Let $B=\{y_j^x\}$. We have
\[
  \sharp B \leq 3^{k_a n+2} s(n,\frac{\eta}{3},F).
\]

Finally, let $y\in M$. There exists $R_0>0$ independent of $y$ such that $\varphi_t(y)\in T$ for some $t\in (-R_0,R_0)$. So there exists $x\in A$ and $j\leq m(x)$ for which $\varphi_t(y)\in A(x)\cap \iota_x K_j(x)$. Thus, by ($*$), $\tilde\varphi_{t+i}(y)$ and $\varphi_i(y^x_j)$ belong to the same cylinder $C_{n(i)}(x)$, and
\[
  d(\varphi_{t+i}(y),\varphi_i(y^x_j))\leq \frac{2\varepsilon}{1-a^2}+2\eta\textrm{ and } d(\varphi_i(y^x_j),\varphi_{t+i}(y))\leq 2\varepsilon+\frac{2\varepsilon}{1-a^2},
\]
for all $-[\frac{n\tau}{1-a}]\leq i\leq [\frac{n\tau}{1+a}]$. Moreover, there exists a constant $\omega$ such that for small $\eta$ and any $z,z'\in M$  the inequalities $\max\{d(z,z')<\eta,d(z',z)<\eta\}$ implies the relations $\max\{d(\varphi_t(z),\varphi_t(z')), d(\varphi_t(z'),\varphi_t(z))\}\leq \omega\eta$ for all $t\in[-R_0,R_0]$. Therefore, the set $\varphi_{-[\frac{\tau n}{1-a}]} B$ is $(\frac{2\tau n}{1-a^2}, 2\omega(\eta+\frac{\epsilon}{1-a^2}))$-spanning with respect to $\varphi$. Next,
\[
  r\left(\frac{2n\tau}{1-a^2}, 2\omega \left(\eta+\frac{\epsilon}{1-a^2}\right),\phi\right) \leq 3^{k_a n+2}s\left(n,\frac{\eta}{3},F\right).
\]
Thus
\[
  \frac{2\tau}{1-a^2} r\left(2\omega \left(\eta+\frac{\epsilon}{1-a^2}\right),\phi\right) \leq  k_a \log 3 + s\left(\frac{\eta}{3},F\right).
\]
This gives
\[
  \frac{2}{1-a^2} \htop(\varphi) \leq k_a\cdot \log 3 + h(\f,F),
\]
when we tend with $\eta$ and $\epsilon$ to zero, and choose $\tau$ arbitrarily close to $1$. Replacing $F$ by $\lambda F$, $\lambda >0$ we must replace $X$ by $\lambda^{-1} X$, and $(\varphi_t)$ by $(\tilde \varphi_t) = (\varphi_{t/\lambda})$. Hence
\[
  \frac{2}{1-a^2} \htop(\varphi) \leq \lambda \log 3 +h(\f,F).
\]
Since $\lambda$ can be arbitrarily small, we get the equality.
\end{proof}

\begin{rem}
 Let $\f$ be a one-dimensional foliation given by a vector field $X$. If $F$ is  Riemannian, then we get the exact result as in Theorem 3.4.3 of \cite{bib:Wal04}, that is,
 \[
  h(\f,F) = 2\htop(\varphi).
 \]
\end{rem}

\begin{rem}
The 1-form $\beta$ in the Randers metric is commonly understood as a mild wind blowing along the leaves of foliation. The direct conclusion of Theorem \ref{thm:constantWind} is that the increasing of a wind along the leaves increases the entropy $h(\f,F)$. 
\end{rem}

\end{document}